\documentclass[12pt]{article}
\usepackage[centertags]{amsmath}
\usepackage{amsfonts}
\usepackage{amssymb}
\usepackage{amsthm}
\usepackage{multicol}

\usepackage{colortbl} 
\usepackage{pgf,tikz}
\usetikzlibrary{arrows}

\definecolor{mygray}{rgb}{0.75,0.75,0.75}

\theoremstyle{plain}
\newtheorem{thm}{Theorem}[section]
\newtheorem{cor}[thm]{Corollary}
\newtheorem{lem}[thm]{Lemma}
\newtheorem{prop}[thm]{Proposition}
\newtheorem{remark}[thm]{Remark}
\newtheorem{conj}[thm]{Conjecture}

\theoremstyle{definition}
\newtheorem{defn}[thm]{Definition}
\newtheorem{exam}[thm]{Example}

\theoremstyle{remark}
\numberwithin{equation}{section}

\newcommand{\beast}{\begin{eqnarray*}}
\newcommand{\eeast}{\end{eqnarray*}}

\title{A construction of group divisible designs with zero block sum}

\author{
Chong-Dao Lee\footnote{Department of Communication Engineering, I-Shou University, Kaohsiung, Taiwan}
\and
Yaotsu Chang\footnote{Department of Financial and Computational Mathematics, I-Shou University, Kaohsiung, Taiwan}
\and
Chia-an Liu\footnote{Corresponding author. E-mail address: liuchiaan8@gmail.com (C.-a. Liu).}
~\footnote{Department of Financial and Computational Mathematics, I-Shou University, Kaohsiung, Taiwan}}

\date{December 30, 2016}

\begin{document}
\maketitle

\bibliographystyle{plain}

\bigskip

\begin{abstract}
This paper gives a construction of group divisible designs on the binary extension fields with block sizes $3,$ $4,$ $5,$ $6,$ and $7,$ respectively,
which is motivated from the decoding of binary quadratic residue codes.
A conjecture is proposed for this construction of group divisible designs with larger block sizes.

\bigskip

\end{abstract}

\section{Introduction}   \label{sec_intro}
Assmus and Mattson in 1969 \cite{am:69} first proposed balanced incomplete block designs (BIBDs) via the theory of error-correcting codes. The codewords of any fixed weight in an extended quadratic residue code \cite{am:69} (respectively, a Reed-Muller code \cite{cd:07}, an extremal binary doubly-even self-dual code \cite{cd:07}, and a Pless symmetry code \cite{p:72}) form a $2$-design (respectively, $3$-design, $5$-design, and $5$-design). The minimum weight codewords in a linear perfect code \cite{am:74} with minimum distance $d=2e+1$ support an $(e+1)$-design. It was shown in \cite{kp:94} that the codewords of any fixed weight in two codes, an extremal binary even formally self-dual code and its dual code, forms a $3$-design. For more $t$-designs supported by other error-correcting codes, the reader is referred to \cite{bg:04}. From the above results, the codewords of error-correcting codes play a significant role in constructing BIBDs.
In the theoretical aspect, the study on $t$-designs over finite fields~\cite{bkow:14,flv:14} also gets some attention.

\medskip

The group divisible design (GDD) is a topic generalized from the pairwise balanced design (well-known as PBD) \cite[Definition~1.4.1]{a:90}. Since GDD has been widely applied to graphs \cite{fr:98} and matrices \cite{ss:98}, many authors proposed different constructions of a GDD. One can see \cite{fr:98,ss:98,hs:04}, \cite[Definition~1.4.2]{a:90}, \cite[Definition~7.14]{s:04} and \cite[Definition~5.5]{w:09} for some examples. Recently, GDDs have been used in the constructions of optical orthogonal codes \cite{wy:10} \cite{wc:15}, constant-weight codes \cite{clls:07} \cite{cgl:08}, and constant-composition codes \cite{cl:07}. However, there are very few studies focused on GDDs constructed from error-correcting codes.

\medskip

In 2003, Chang \emph{et al.} \cite{ctrcl:03} developed the new decoders for three binary quadratic residue codes with irreducible polynomials.
Motivated by the decoding of binary quadratic residue codes, this study considers the problem of constructing GDD. A \emph{group divisible design} $\text{GDD}(v,n,k)$ is a triple ($X,\mathcal{G},\mathcal{B}$), where $\mathcal{G}$ is a collection of $n$-subsets of $v$-set $X$ and $\mathcal{B}$ is a collection of $k$-subsets of $X.$ In this paper, we assume $X={\Bbb F}_{2^m}\setminus \{0,1\}$ and consider the correctable error patterns $(x_1,~x_2,\ldots,x_k)$ with a fixed weight $k$ and satisfying $\alpha^{x_1}+\alpha^{x_2}+\cdots+\alpha^{x_k}=1$ in the finite field $\Bbb F_{2^m}$, where distinct integers $1\leq x_i\leq 2^m-1$ for $1\leq i\leq k\leq m$ and $\alpha$ is a primitive element of $\Bbb F_{2^m}$. If $k=2$, then those error patterns form a group set $\mathcal G$. Similarly, for each $3\leq k\leq m,$ these error patterns support a block set $\mathcal{B}$. This paper gives a construction of group divisible designs with block sizes $3,$ $4,$ $5,$ $6,$ and $7,$ respectively. The correctness and parameters of the construction are obtained by using the inclusion-exclusion principle.

\medskip

\medskip

The paper is organized as follows. Preliminary notations are introduced in Section~\ref{sec_pre}. The details of our construction of GDDs are proposed in Section~\ref{sec_GDD}. Section~\ref{sec_con} summarizes the results obtained from Section~\ref{sec_GDD} and presents a conjecture for group divisible designs with larger block sizes.


\section{Preliminary}   \label{sec_pre}

Basic results of the group divisible design and finite field are provided in this section for later used. The notations and definitions of a GDD can be referred to \cite[Definition~1.4.2]{a:90}.
\begin{defn}\label{defn_GDD}
A \emph{group divisible design} $\text{GDD}(v,n,k)$ is a triple ($X,\mathcal{G},\mathcal{B}$), where $\mathcal{G}$ is a collection of $n$-subsets of $v$-set $X$ and $\mathcal{B}$ is a collection of $k$-subsets of $X.$ We say that $\mathcal{G}$ is the \emph{group set} and each element in $\mathcal{G}$ is a \emph{group}, and $\mathcal{B}$ is the \emph{block set} and each element in $\mathcal{B}$ is a \emph{block}, such that:
\begin{enumerate}
\item[(i)] $\mathcal{G}$ forms a partition of $X,$
\item[(ii)] for all $B\in \mathcal{B}$ and $u,v\in B$ there does not exist $G\in\mathcal{G}$ such that ${u,v}\in G,$ and
\item[(iii)] every pair of distinct elements $x$ and $y$ from different groups occur together in exactly $\lambda$ blocks.
\end{enumerate}
In particular, the condition (iii) is called the \emph{balance condition}, and $\lambda$ is called the \emph{balance parameter} of ($X,\mathcal{G},\mathcal{B}$).
\end{defn}

\medskip

Let $(X,\mathcal{G},\mathcal{B})$ be a GDD and $r_x$ denote the number of blocks in $\mathcal{B}$ that contain $x$ for each $x\in X.$ The following result given in~\cite[Proposition~2.2]{t:15} tells that $r_x$ is independent of the choice of $x$ which is called the \emph{repetition number} of $(X,\mathcal{G},\mathcal{B}).$
\begin{prop}    \label{prop_r_const}
Let ($X,\mathcal{G},\mathcal{B}$) be a $\text{GDD}(v,n,k)$ with balance parameter $\lambda.$ Then each element in $X$ occurs in
\begin{equation}    \label{eq_r}
r=\frac{\lambda(v-n)}{k-1}
\end{equation}
blocks.
\end{prop}

\medskip

Let $r=r_x$ be the repetition number of $(X,\mathcal{G},\mathcal{B}).$ Since each block in $\mathcal{B}$ is of cardinality $k,$ one can get the number of blocks in $\mathcal B$, denoted by $b=|\mathcal{B}|$, by direct counting method.
\begin{prop}    \label{prop_num_of_block}
Let ($X,\mathcal{G},\mathcal{B}$) be a $\text{GDD}(v,n,k)$ with balance parameter $\lambda$ and repetition number $r.$ Then, the number of blocks in $\mathcal B$ is
\begin{equation}    \label{eq_b}
b=|\mathcal{B}|=\frac{vr}{k}.
\end{equation}
\qed
\end{prop}

\medskip

The finite field properties in the following are referred to \cite[Sec 4.2]{p:98}.
\begin{itemize}
\item [(i)] Every finite field has $p^{m}$ elements for some prime $p$ and positive integer $m.$
\item [(ii)] For any positive integer $m,$ there is a unique field (up to isomorphism) of $2^m$ elements. We denote this field by $\mathbb{F}_{2^m}.$
\item[(iii)] The multiplicative group $\mathbb{F}_{2^m}\setminus\{0\}$ is cyclic so that there exists a generator (which is so-called a \emph{primitive element}) of $\mathbb{F}_{2^m}\setminus\{0\}.$
\end{itemize}
Finite field is an important topic in Abstract Algebra. See \cite[Chapter~6]{w:12} for more details.

\medskip

Throughout this paper, one considers $X=\mathbb{F}_{2^m}\setminus\{0,1\}$, where $0$ and $1,$ respectively, denote the zero and unity elements in the finite field $\mathbb{F}_{2^m}$ of order $2^m$ for arbitrary positive integer $m\geq 3.$ Note that the cardinality of $X$ is $|X|=2^{m}-2.$


\section{A construction of group divisible designs}   \label{sec_GDD}

The aim of this section is to propose a construction of group divisible designs with block sizes $3,4,5,6,$ and $7$. The collections $W_k$ of $k$-subsets of $X$ for $k \geq 2$ are given in the following. It will be verified that $W_2$ forms a partition of $X,$ and for each $3\leq k \leq 7$ a GDD with block size $k$ is constructed from $W_k.$
\begin{defn}    \label{defn_Wk}
For each positive integers $k \geq 2,$ let
$$W_k=\{B\subset X~\mid~|B|=k, \sum\limits_{i \in B}i=1,~\text{and}~{B \choose \ell}\cap W_\ell=\phi~~~\text{for all~} 2\leq\ell\leq k-3\}.$$
\end{defn}

\medskip

The next observations are directly from Definition~\ref{defn_Wk}.
\begin{remark}      \label{rem_Wk}
The condition ${B \choose \ell}\cap W_\ell=\phi~\text{for all~}2\leq\ell\leq k-3$ in Definition~\ref{defn_Wk} can be realized as $\text{for all~} 2\leq\ell\leq k-1.$ From the condition $|B|=k,$ if $x \in {B \choose k-1}\cap W_{k-1},$ then the only element in $B\setminus \{x\}$ is $0.$ However, $0 \not\in B \subset X,$ which is a contradiction. If $y \in {B \choose k-2}\cap W_{k-2},$ then the sum of $B \setminus \{x\}$ is $0$ so that the two elements in $B \setminus \{x\}$ are equal, which also contradicts to $|B|=k.$
\end{remark}

\medskip

An example of $W_2$ and $W_3$ is illustrated.
\begin{exam}    \label{exam_m3}
Let $m=3$. Let $\gamma=x$ be a primitive element of the finite field $\mathbb{F}_{2^3}\cong \mathbb{F}_2[x]/\langle x^3+x+1\rangle$.  Then, one has $\gamma^2=x^2,$ $\gamma^3=x+1,$ $\gamma^4=x^2+x,$ $\gamma^5=x^2+x+1,$ and $\gamma^6=x^2+1.$ Let $X=\{\gamma^i\mid i=1,2,\ldots,6\}.$ From Definition \ref{defn_Wk}, the collection $W_2$ of $2$-subsets of $X$ is
\begin{equation}
W_2=\{\{\gamma,\gamma^3\},\{\gamma^2,\gamma^6\},\{\gamma^4,\gamma^5\}\}
\nonumber
\end{equation}
which forms a partition of $X,$ and the collection $W_3$ of $3$-subsets of $X$ is
\begin{equation}
W_3=\{\{\gamma,\gamma^2,\gamma^5\},\{\gamma,\gamma^6,\gamma^4\}\},
\{\gamma^3,\gamma^2,\gamma^4\},\{\gamma^3,\gamma^6,\gamma^5\}\},
\nonumber
\end{equation}
where each block $B\in W_3$ is with cardinality $|B|=3$ and block sum $\sum_{i\in B}i=1$ in $\mathbb{F}_{2^3}.$
\end{exam}

\medskip

Two results are shown below that the collection $W_2$ forms a partition of $X$ and each block in $W_2$ is not a subset of $W_k$ for $k \geq 3,$ so $W_2$ forms a group set for constructing GDD with respect to $X.$
\begin{lem}     \label{lem_W2}
The blocks set $W_2$ forms a partition of $X$ and the number of blocks in $W_2$ is $\frac{2^m-2}{2}.$
\end{lem}
\begin{proof}
For each $a \in X,$ $0,1 \not\in X$ implies $a \not\in\{0,1\},$ so $a+1:=b \not\in\{0,1\}$ either. Hence, $b\in X.$ Besides, $a\neq b$ since $a+b=1\neq 0.$ Therefore, $\{a,b\}\in W_2$ and $W_2$ forms a partition of $X.$ Then, the number of blocks in $W_2$ is counted by
$$|W_2|=\frac{|X|}{2}=\frac{|\mathbb{F}_{2^m}\setminus\{0,1\}|}{2}=\frac{2^m-2}{2}.$$
\end{proof}

\medskip

\begin{lem} \label{lem_W2_group}
For each $k \geq 3,$ a block in $W_2$ is not a subset of any block in $W_k.$
\end{lem}
\begin{proof}
The result immediately follows from Definition~\ref{defn_GDD}.
\end{proof}

It should be noticed that Lemmas~\ref{lem_W2} and~\ref{lem_W2_group} give the conditions in Definition~\ref{defn_GDD}~(i) and (ii), respectively, for the triple $(X,W_2,W_k).$ Next, in order to prove that the triple $(X,W_2,W_k)$ is a $\text{GDD}(2^m-2,2,k)$ for positive integers $m\geq k\geq 3,$ it is sufficient to find a balance parameter $\lambda_k.$

\medskip

First of all, a group divisible design with block size $k=3$ is presented.
\begin{thm}     \label{thm_W3}
For $m \geq 3,$ the triple $(X,W_2,W_3)$ is a $\text{GDD}(2^m-2,2,3)$ with balance parameter
$$\lambda_3=1.$$
\end{thm}
\begin{proof}
Given two distinct elements $u,v \in X=\mathbb{F}_{2^m}\setminus\{0,1\}$ with $\{u,v\}\not\in W_2.$ Let $\lambda_3(u,v)$ be the number of blocks in $W_3$ that contains both $u$ and $v.$ Then, by Lemmas~\ref{lem_W2} and~\ref{lem_W2_group}, it suffices to prove that $\lambda_3(u,v)=1,$ which is independent of the choice of $u$ and $v.$ By letting $k=3$ in Definition~\ref{defn_Wk}, one can see that the only block in $W_3$ that contains $u$ and $v$ is $\{u,v,u+v+1\}.$ Note that $u+v+1 \in \mathbb{F}_{2^m}\setminus\{0,1,u,v\}$ since $u,v$ are two distinct elements in $X$ with $\{u,v\}\not\in W_2.$ The result follows.
\end{proof}

\medskip

Substituting $v=2^m-2,$ $n=2,$ and $k=3$ into~\eqref{eq_r} in Proposition~\ref{prop_r_const} and~\eqref{eq_b} in Proposition~\ref{prop_num_of_block}
gives
$$r_3=\lambda_3\cdot\frac{(2^m-2)-2}{3-1}=\lambda_3\cdot\frac{2^m-4}{2}$$
and
$$b_3=r_3\cdot\frac{2^m-2}{3},$$
respectively.
Therefore Corollary~\ref{cor_W3} follows.
\begin{cor}     \label{cor_W3}
The repetition number of the triple $(X,W_2,W_3)$ is
$$r_3=\frac{2^m-4}{2!}$$
and the number of blocks in $W_3$ is
$$|W_3|=b_3=\frac{(2^m-2)(2^m-4)}{3!}.$$
\qed
\end{cor}

\medskip

Several blocks sets, which will be used in the proofs of the theorems, are defined.
\begin{defn}    \label{defn_omega_tau}
Given two distinct elements $u,v\in X$ with $\{u,v\}\not\in W_2.$ Let $z=u+v$ and $S=\{u,v,u+1,v+1\}.$ For each $k \geq 3,$ define the blocks sets
\begin{eqnarray}
\Omega_{z,k}&=&\{B\in W_k~|~z \in B\},
\nonumber \\
\omega_{\alpha,k}&=&\{B\in \Omega_{z,k}~|~\alpha \in B\setminus\{z\}\}~~~\text{for each}~\alpha\in S,~~~\text{and}
\nonumber \\
\tau_{\alpha,k}&=&\{B\in \Omega_{z,k}~|~\exists a,b\in B\setminus\{z\}\text{~such that~}a+b=\alpha\}~~~\text{for each}~\alpha\in S.
\nonumber
\end{eqnarray}
\end{defn}

\medskip

Below, Example~\ref{exam_m3} is reviewed in order to realize the blocks sets defined in Definition~\ref{defn_omega_tau}.
\begin{exam}       \label{exam_m3_2}
As stated in Example~\ref{exam_m3}, $X=\{\gamma^i\mid i=1,2,\ldots,6\},$ where the elements are defined as $\gamma=x,$ $\gamma^2=x^2,$ $\gamma^3=x+1,$ $\gamma^4=x^2+x,$ $\gamma^5=x^2+x+1,$ and $\gamma^6=x^2+1.$ The collection $W_3$ of $3$-subsets of $X$ is
\begin{equation}
W_3=\{\{\gamma,\gamma^2,\gamma^5\},\{\gamma,\gamma^6,\gamma^4\},
\{\gamma^3,\gamma^2,\gamma^4\},\{\gamma^3,\gamma^6,\gamma^5\}\}.
\nonumber
\end{equation}
Let $u=\gamma$ and $v=\gamma^2.$ Thus, $u+1=\gamma^3,$ $v+1=\gamma^6,$ and $z=u+v=x+x^2=\gamma^4.$ Finally, five blocks subsets of $W_3$ are
\begin{eqnarray}
\Omega_{z,3}&=&\{\{\gamma,\gamma^6,\gamma^4\}\},\{\gamma^3,\gamma^2,\gamma^4\},\{\gamma^3,\gamma^6,\gamma^5\}\}
\nonumber   \\
\omega_{u,3}&=&\omega_{v+1,3}=\{\{\gamma,\gamma^6,\gamma^4\}\},~~~~~~~~\text{and}
\nonumber   \\
\omega_{v,3}&=&\omega_{u+1,3}=\{\{\gamma^3,\gamma^2,\gamma^4\}\}.
\nonumber
\end{eqnarray}
\end{exam}

\medskip

Example~\ref{exam_m3_2} illustrates the case for $m=k=3,$ and it is easy to see that $\tau_{\alpha,3}=\phi$ for $\alpha\in S.$ An example for the case $m=k=4$ is further presented.
\begin{exam}    \label{exam_m4}
Let $m=4$. Let $\gamma=x$ be a primitive element of the finite field $\mathbb{F}_{2^4}\cong \mathbb{F}_2[x]/\langle x^4+x+1\rangle$.
Then, $X=\{\gamma^i\mid i=1,2,\ldots,14\},$ where the elements are presented as follows:
\begin{center}
\begin{tabular}{l|l}
  $i$ & $\gamma^i$ \\
  \hline
  2 & $x^2$ \\
  3 & $x^3$ \\
  4 & $x+1$ \\
  5 & $x^2+x$ \\
  6 & $x^3+x^2$ \\
  7 & $x^3+x+1$ \\
\end{tabular}
~~~~~~~~~~~
\begin{tabular}{l|l}
  8 & $x^2+1$ \\
  9 & $x^3+x$ \\
  10 & $x^2+x+1$ \\
  11 & $x^3+x^2+x$ \\
  12 & $x^3+x^2+x+1$ \\
  13 & $x^3+x^2+1$ \\
  14 & $x^3+1$ \\
\end{tabular}.
\end{center}
Let $u=\gamma$ and $v=\gamma^2.$ It is not difficult to check that $u+1=\gamma^4,$ $v+1=\gamma^8,$ and $z=u+v=x+x^2=\gamma^5.$ For $\alpha,\beta\in S=\{u,v,u+1,v+1\},$ the blocks subsets $\Omega_{z,4},$ $\omega_{\alpha,4},$ and $\tau_{\beta,4}$ of $W_4$ can be written as
\begin{eqnarray}
\Omega_{z,4}&=&\left\{\begin{array}{ccc}
                        \{\gamma^5,\gamma,\gamma^3,\gamma^{13}\}, &
                        \{\gamma^5,\gamma,\gamma^6,\gamma^{14}\}, &
                        \{\gamma^5,\gamma,\gamma^7,\gamma^{11}\}, \\ \{\gamma^5,\gamma,\gamma^9,\gamma^{12}\}, &
                        \{\gamma^5,\gamma^2,\gamma^3,\gamma^7\}, & \{\gamma^5,\gamma^2,\gamma^6,\gamma^{12}\}, \\ \{\gamma^5,\gamma^2,\gamma^9,\gamma^{14}\}, & \{\gamma^5,\gamma^2,\gamma^{11},\gamma^{13}\}, &
                        \{\gamma^5,\gamma^3,\gamma^4,\gamma^6\}, \\
                        \{\gamma^5,\gamma^3,\gamma^8,\gamma^9\}, & \{\gamma^5,\gamma^4,\gamma^7,\gamma^{12}\}, & \{\gamma^5,\gamma^4,\gamma^9,\gamma^{11}\}, \\
                        \{\gamma^5,\gamma^4,\gamma^{13},\gamma^{14}\}, & \{\gamma^5,\gamma^6,\gamma^8,\gamma^{11}\}, & \{\gamma^5,\gamma^7,\gamma^8,\gamma^{14}\}, \\ \{\gamma^5,\gamma^8,\gamma^{12},\gamma^{13}\} &&
                      \end{array}
\right\},
\nonumber   \\
\omega_{u,4}&=&\tau_{v+1,4}
\nonumber   \\
&=&\{\{\gamma^5,\gamma,\gamma^3,\gamma^{13}\},\{\gamma^5,\gamma,\gamma^6,\gamma^{14}\},
\{\gamma^5,\gamma,\gamma^7,\gamma^{11}\},\{\gamma^5,\gamma,\gamma^9,\gamma^{12}\}\},
\nonumber   \\
\omega_{v,4}&=&\tau_{u+1,4}
\nonumber   \\
&=&\{\{\gamma^5,\gamma^2,\gamma^3,\gamma^7\},\{\gamma^5,\gamma^2,\gamma^6,\gamma^{12}\},
\{\gamma^5,\gamma^2,\gamma^9,\gamma^{14}\},\{\gamma^5,\gamma^2,\gamma^{11},\gamma^{13}\}\},
\nonumber   \\
\omega_{u+1,4}&=&\tau_{v,4}
\nonumber   \\
&=&\{\{\gamma^5,\gamma^4,\gamma^7,\gamma^{12}\},\{\gamma^5,\gamma^3,\gamma^4,\gamma^6\},
\{\gamma^5,\gamma^4,\gamma^7,\gamma^{12}\},\{\gamma^5,\gamma^4,\gamma^{13},\gamma^{14}\}\},
\nonumber   \\
\text{and}
\nonumber   \\
\omega_{v+1,4}&=&\tau_{u,4}
\nonumber   \\
&=&\{\{\gamma^5,\gamma^3,\gamma^8,\gamma^9\},\{\gamma^5,\gamma^6,\gamma^8,\gamma^{11}\},
\{\gamma^5,\gamma^7,\gamma^8,\gamma^{14}\},\{\gamma^5,\gamma^8,\gamma^{12},\gamma^{13}\}\}.
\nonumber
\end{eqnarray}
\end{exam}

\medskip

The detailed results of the blocks subsets of $W_k$ are provided in the following lemma.
\begin{lem}     \label{lem_omega_tau}
The relations between the blocks sets in Definition~\ref{defn_omega_tau} are described below.
\begin{itemize}
\item[(i)]  $\omega_{\alpha,k},\tau_{\beta,k} \subset \Omega_{z,k}$ for $\alpha,\beta\in S.$
\item[(ii)]    $\omega_{u,3}=\omega_{v+1,3}$ and $\omega_{v,3}=\omega_{u+1,3}.$ If $k\geq 4,$ then $\omega_{\alpha,k}\cap\omega_{\beta,k}=\phi$ for distinct $\alpha,\beta\in S.$
\item[(iii)]     $\omega_{u,4}=\tau_{v+1,4}$ and $\omega_{v,4}=\tau_{u+1,4}.$ If $k \geq 5,$ then $\omega_{\alpha,k}\cap\tau_{\beta,k}=\phi$ for $\alpha,\beta\in S.$
\item[(iv)]     $\tau_{u,5}=\tau_{v+1,5}$ and $\tau_{v,5}=\tau_{u+1,5}.$ If $k\geq 6,$ then $\tau_{\alpha,k}\cap\tau_{\beta,k}=\phi$ for distinct $\alpha,\beta\in S.$
\end{itemize}
\end{lem}
\begin{proof}
(i) is directly from Definition~\ref{defn_omega_tau}.

\smallskip

To prove (ii), it follows from Definition~\ref{defn_omega_tau} that $\omega_{\alpha,3}=\omega_{\alpha+u+v+1,3}$ since $z+u+v+1=1.$ For $k\geq 4,$ suppose to the contrary that there exists $B\in \omega_{\alpha,k}\cap\omega_{\beta,k}$ for some distinct $\alpha,\beta\in S.$ Then, $\alpha,\beta,z\in B.$ Note that $\alpha+\beta\in\{1,u+v,u+v+1\}.$ If $\alpha+\beta=1$, then it contradicts to the definition that ${B \choose 2}\cap W_2=\phi.$ If $\alpha+\beta\in \{u+v,u+v+1\}$, then $\alpha+\beta+z\in\{0,1\}$ which also contradicts to Remark~\ref{rem_Wk}.

\smallskip

To verify (iii), Definition~\ref{defn_omega_tau} indicates that $\omega_{\alpha,4}=\tau_{\alpha+u+v+1,4}$ since $z+u+v+1=1.$ For $k \geq 5,$ suppose to the contrary that there exists $B\in \omega_{\alpha,k}\cap\tau_{\beta,k}$ for some $\alpha,\beta\in S.$ Let distinct $a,b\in B$ such that $a+b=\beta.$ Assume that $a,b$ are both not $\alpha.$ Then, $\alpha+a+b\in\{0,1,u+v,u+v+1\}.$ If $\alpha+a+b\in\{0,1\}$ then it contradicts to Remark~\ref{rem_Wk}. If $\alpha+a+b\in\{u+v,u+v+1\}$, then $z+\alpha+a+b\in\{0,1\},$ which also contradicts to Remark~\ref{rem_Wk}. Thus, without loss of generality, suppose $a=\alpha.$ Then, $b=\alpha+\beta.$ Since $b\not\in \{0,1\},$ $b\in\{u+v,u+v+1\}.$ However, it is easily seen that $b+z\in\{0,1\},$ which contradicts to Remark~\ref{rem_Wk}.

\smallskip

To prove (iv), it is directly from Definition~\ref{defn_omega_tau} that $\tau_{\alpha,5}=\tau_{\alpha+u+v+1,5}$ since $z+u+v+1=1.$ For $k\geq 6,$ suppose to the contrary that there exists $B\in \tau_{\alpha,k}\cap\tau_{\beta,k}$ for some distinct $\alpha,\beta\in S.$ Let distinct $a,b \in B$ and distinct $c,d\in B$ such that $a+b=\alpha$ and $c+d=\beta.$ Assume that $\{a,b\}\cap\{c,d\}=\phi.$ Then, $a+b+c+d\alpha+\beta\in\{0,1,u+v,u+v+1\}.$ If $a+b+c+d\in\{0,1\}$ then it contradicts to Remark~\ref{rem_Wk}. If $a+b+c+d=\in\{u+v,u+v+1\}$, then $z+a+b+c+d\in\{0,1\},$ which also contradicts to Remark~\ref{rem_Wk}. Thus, without loss of generality, suppose $a=c.$ Then, $b+d=\alpha+\beta\in\{0,1,u+v,u+v+1\}.$ If $b+d\in\{0,1\}$, then it contradicts to Remark~\ref{rem_Wk}. If $b+d=\in\{u+v,u+v+1\}$, then $z+b+d\in\{0,1\},$ which also contradicts to Remark~\ref{rem_Wk}. The proof of this lemma is complete.
\end{proof}

\medskip

\begin{lem}     \label{lem_omega_tau_card}
Given two distinct elements $u,v \in X$ and $z=u+v.$ Then for each $\alpha\in S=\{u,v,u+1,v+1\},$
\begin{eqnarray}
|\Omega_{z,k}|&=&r_{k}~~~\text{~for~}k\geq 4,
\nonumber   \\
|\omega_{\alpha,k}|&=&\lambda_{k}~~~\text{~for~}k\geq 4,
\nonumber   \\
|\tau_{\alpha,4}|&=&\frac{1}{2}(2^m-2^3)~~~\text{~for~}m\geq 4,
\nonumber   \\
|\tau_{\alpha,5}|&=&\frac{1}{4}(2^m-2^3)(2^m-2^4)~~~\text{~for~}m\geq 5,~~~\text{and}
\nonumber   \\
|\tau_{\alpha,6}|&=&\frac{1}{12}(2^m-2^3)(2^m-2^4)(2^m-2^5)~~~\text{~for~}m\geq 6.
\nonumber
\end{eqnarray}
\end{lem}
\begin{proof}
Note that $|\omega_{\alpha,k}|=|\omega_{\beta,k}|$ and $|\tau_{\alpha,k}|=|\tau_{\beta,k}|$ for any $\alpha,\beta\in S$ because of the symmetry. Fixed some $\alpha\in S,$ the cardinalities of $\Omega_{z,k}$ and $\omega_{\alpha,k}$ are from the definition of the repetition number $r_k$ and balance parameter $\lambda_k,$ respectively.

\smallskip

To count $|\tau_{\alpha,4}|,$ let $B\in \tau_{\alpha,4}$ such that $B=\{z,a,\alpha+a,z+\alpha+1\}$ without loss of generality. Note that $B\subset X=\mathbb{F}_{2^m}\setminus\{0,1\}$ and by Remark~\ref{rem_Wk}, ${B \choose 2}\cap W_2={B \choose 3}\cap W_3=\phi.$ Hence $a$ can be chosen from
$$\mathbb{F}_{2^m}\setminus\left(\{0,1\}+\{0,u\}+\{0,v\}\right)
=\mathbb{F}_{2^m}\setminus\{0,1,u,u+1,v,v+1,u+v,u+v+1\}$$
where the addition $+$ between two subsets $A,B$ of $\mathbb{F}_{2^m}$ is defined as $A+B=\{i+j~|~i\in A~\text{and}~j\in B\}.$ Since the two elements $a$ and $\alpha+a$ are not ordered, there are $\frac{2^m-2^3}{2}$ ways to determine $B,$ which implies $|\tau_{\alpha,4}|=\frac{1}{2}(2^m-2^3).$

\smallskip

To count $|\tau_{\alpha,5}|,$ let $B\in \tau_{\alpha,5}$ such that $B=\{z,a,\alpha+a,b,z+\alpha+1+b\}$ without loss of generality. There are $\frac{2^m-2^3}{2}$ ways to determine the elements $a$ and $\alpha+a$ from the argument of counting $|\tau_{\alpha,4}|.$ Note that $B\subset X$ and by Remark~\ref{rem_Wk} we have ${B \choose 2}\cap W_2={B \choose 3}\cap W_3={B \choose 4}\cap W_4=\phi.$ Hence, $b$ can be chosen from
$$\mathbb{F}_{2^m}\setminus\left(\{0,1\}+\{0,u\}+\{0,v\}+\{0,a\}\right).$$
Since the two elements $b$ and $z+\alpha+1+b$ are not ordered, there are $\frac{2^m-2^4}{2}$ ways to determine them, which implies $$|\tau_{\alpha,5}|=\frac{1}{2}(2^m-2^3)\cdot\frac{1}{2}(2^m-2^4)=\frac{1}{4}(2^m-2^3)(2^m-2^4).$$

\smallskip

To count $|\tau_{\alpha,6}|,$ let $B\in \tau_{\alpha,6}$ such that $B=\{z,a,\alpha+a,b,c,z+\alpha+1+b+c\}$ without loss of generality. There are $\frac{2^m-2^3}{2}$ ways to determine the elements $a$ and $\alpha+a$ from the argument of counting $|\tau_{\alpha,4}|.$ Similar with the lower part of counting $|\tau_{\alpha,5}|,$ there are $(2^m-2^4)$ ways to pick $b,$ and then $(2^m-2^5)$ ways to pick $c.$ Since the three elements $b,$ $c$ and $z+\alpha+1+b+c$ are not ordered, there are $\frac{1}{3!}(2^m-2^4)(2^m-2^5)$ ways to determine them, which implies $$|\tau_{\alpha,6}|=\frac{1}{2}(2^m-2^3)\cdot\frac{1}{3!}(2^m-2^4)(2^m-2^5)=\frac{1}{12}(2^m-2^3)(2^m-2^4)(2^m-2^5).$$
\end{proof}

\medskip

The blocks sets introduced in Definition~\ref{defn_omega_tau} will be used to construct group divisible designs with block size $4,5,6$ and $7.$

We are now ready to present a GDD with block size $4$. The cardinalities $|\Omega_{z,3}|$ and $|\omega_{\alpha,3}|$ found in Lemma~\ref{lem_omega_tau_card} help in counting the parameter $\lambda_4.$
\begin{thm}     \label{thm_W4}
If $m \geq 4,$ then the triple $(X,W_2,W_4)$ is a $\text{GDD}(2^m-2,2,4)$ with balance parameter
$$\lambda_4=\frac{2^m-8}{2}.$$
\end{thm}
\begin{proof}
Given two distinct elements $u,v \in X$ with $\{u,v\}\not\in W_2.$ Let $\lambda_4(u,v)$ be the number of blocks in $W_4$ that contains both $u$ and $v.$ Let $z=u+v$ and $S=\{u,v,u+1,v+1\}.$ The blocks sets $\Omega_{z,3}$ and $\omega_{\alpha,3}$ for $\alpha\in S$ are mentioned in Definition~\ref{defn_omega_tau}.

Note that a block $B\in W_4$ that contains both $u$ and $v$ corresponds to a unique block $\overline{B} \in \Omega_{z,3}$ such that $B\setminus \{u,v\}=\overline{B}\setminus\{z\}.$ However, according to the above corresponding rule, for each block $B \in W_4,$ $|B|=4$ and ${B \choose 2}\cap W_2=\phi$ imply $S\cap\overline{B}=\phi.$
Applying Lemma~\ref{lem_omega_tau}~(ii) and the cardinalities $|\Omega_{z,3}|,$ $|\omega_{u,3}|$ given in Lemma~\ref{lem_omega_tau_card} yields
$$\lambda_4(u,v)=|\Omega_{z,3}|-2|\omega_{u,3}|=r_3-2\lambda_3.$$
The Venn diagram for $\Omega_{z,3}$ is shown in Figure~1. Moreover, from the values of $\lambda_3$ in Theorem~\ref{thm_W3} and $r_3$ in Corollary~\ref{cor_W3}, one has
$$\lambda_4(u,v)=\frac{2^m-4}{2!}-2=\frac{2^m-8}{2!},$$
which is independent of the choice of $x$ and $y.$ The desired conclusion follows.
\end{proof}

\begin{center}
\begin{tikzpicture}
\fill[mygray] (0,0) rectangle (6,4);
\fill[white] (1.8,2) circle (0.8)
             (4.2,2) circle (0.8);
\draw (1.8,2) circle (0.8) (1.8,2.8)  node [text=black,above] {$\omega_{u,3}=\omega_{v+1,3}$}
      (4.2,2) circle (0.8) (4.2,2.8)  node [text=black,above] {$\omega_{v,3}=\omega_{u+1,3}$}
      (0,0) rectangle (6,4) node [text=black,above] {$\Omega_{z,3}$};
\end{tikzpicture}\\
\bigskip
{\bf Figure 1:} The Venn diagram for the proof of Theorem~\ref{thm_W4}.
\end{center}

\medskip

Substituting $v=2^m-2,$ $n=2,$ and $k=3$ into~\eqref{eq_r} in Proposition~\ref{prop_r_const} and~\eqref{eq_b} in Proposition~\ref{prop_num_of_block}
gives
$$r_4=\lambda_4\cdot\frac{(2^m-2)-2}{4-1}=\lambda_4\cdot\frac{2^m-4}{3}$$
and
$$b_4=r_4\cdot\frac{2^m-2}{4},$$
respectively.
Thus one has Corollary~\ref{cor_W4}.
\begin{cor}     \label{cor_W4}
The repetition number of the triple $(X,W_2,W_4)$ is
$$r_4=\frac{(2^m-4)(2^m-8)}{3!}$$
and the number of blocks in $W_4$ is
$$|W_4|=b_4=\frac{(2^m-2)(2^m-4)(2^m-8)}{4!}.$$
\qed
\end{cor}

\medskip

A group divisible design with block size $5$ is proposed. To count the parameter $\lambda_5,$ the cardinalities $|\Omega_{z,4}|,$ $|\omega_{\alpha,4}|,$ and $|\tau_{\beta,4}|$ obtained in Lemma~\ref{lem_omega_tau_card} are used.
\begin{thm}     \label{thm_W5}
If $m \geq 5,$ then the triple $(X,W_2,W_5)$ is a $\text{GDD}(2^m-2,2,5)$ with balance parameter
$$\lambda_5=\frac{(2^m-8)(2^m-16)}{3!}.$$
\end{thm}
\begin{proof}
First, consider two distinct elements $u,v \in X$ with $\{u,v\}\not\in W_2.$ Let $\lambda_5(x,y)$ be the number of blocks in $W_5$ that contains both $u$ and $v.$ Owing to $z=u+v$ and $S=\{u,v,u+1,v+1\},$ the blocks sets $\Omega_{z,4},$ $\omega_{\alpha,4},$ and $\tau_{\alpha,4}$ for $\alpha\in S$ are known from Definition~\ref{defn_omega_tau}.

It is important to note that a block $B\in W_5$ that contain both $u$ and $v$ corresponds to a unique block $\overline{B} \in \Omega_{z,4}$ such that $B\setminus \{u,v\}=\overline{B}\setminus\{z\}.$ However, for each block $B \in W_5,$ $|B|=5$ implies $u,v \not\in \overline{B},$ and ${B \choose 2}\cap W_2=\phi$ implies $u+1,v+1 \not\in \overline{B}.$
As a consequence of Lemma~\ref{lem_omega_tau}~(iii) and the cardinalities $|\Omega_{z,4}|,$ $|\omega_{u,4}|$ obtained in Lemma~\ref{lem_omega_tau_card}, we have
$$\lambda_5(u,v)=|\Omega_{z,4}|-4|\omega_{u,4}|=r_4-4\lambda_4,$$
where the Venn diagram for $\Omega_{z,4}$ is depicted in Figure~2. In accordance with the values of $\lambda_4$ in Theorem~\ref{thm_W4} and $r_4$ in Corollary~\ref{cor_W4}, the parameter $\lambda_5$ can be further expressed as
$$\lambda_5(u,v)=\frac{(2^m-4)(2^m-8)}{3!}-4\frac{2^m-8}{2!}=\frac{(2^m-8)(2^m-16)}{3!},$$
which is independent of the choice of $u$ and $v.$ The proof of this theorem is completed.
\end{proof}

\begin{center}
\begin{tikzpicture}
\fill[mygray] (0,0) rectangle (6,5);
\fill[white] (1.8,1) circle (0.8)
             (4.2,1) circle (0.8)
             (1.8,3.4) circle (0.8)
             (4.2,3.4) circle (0.8);
\draw (1.8,1) circle (0.8) (1.8,1.8)  node [text=black,above] {$\omega_{u,4}=\tau_{v+1,4}$}
      (4.2,1) circle (0.8) (4.2,1.8)  node [text=black,above] {$\omega_{v,4}=\tau_{u+1,4}$}
      (1.8,3.4) circle (0.8) (1.8,4.2)  node [text=black,above] {$\omega_{u+1,4}=\tau_{v,4}$}
      (4.2,3.4) circle (0.8) (4.2,4.2)  node [text=black,above] {$\omega_{v+1,4}=\tau_{u,4}$}
      (0,0) rectangle (6,5) node [text=black,above] {$\Omega_{z,4}$};
\end{tikzpicture}\\
\bigskip
{\bf Figure 2:} The Venn diagram for the proof of Theorem~\ref{thm_W5}.
\end{center}

\medskip

Substituting $v=2^m-2,$ $n=2,$ and $k=5$ into~\eqref{eq_r} in Proposition~\ref{prop_r_const} and~\eqref{eq_b} in Proposition~\ref{prop_num_of_block}
leads to
$$r_5=\lambda_5\cdot\frac{(2^m-2)-2}{5-1}=\lambda_3\cdot\frac{2^m-4}{4}$$
and
$$b_5=r_5\cdot\frac{2^m-2}{5},$$
respectively.
Hence Corollary~\ref{cor_W5} follows.
\begin{cor}     \label{cor_W5}
The repetition number of the triple $(X,W_2,W_5)$ is
$$r_5=\frac{(2^m-4)(2^m-8)(2^m-16)}{4!}$$
and the number of blocks in $W_5$ is
$$|W_5|=b_5=\frac{(2^m-2)(2^m-4)(2^m-8)(2^m-16)}{5!}.$$
\qed
\end{cor}

\medskip

A group divisible design with block size $6$ is presented. The cardinalities $|\Omega_{z,5}|,$ $|\omega_{\alpha,5}|,$ and $|\tau_{\beta,5}|$ found in Lemma~\ref{lem_omega_tau_card} are applied to calculate the parameter $\lambda_6.$
\begin{thm}     \label{thm_W6}
If $m \geq 6,$ then the triple $(X,W_2,W_6)$ is a $\text{GDD}(2^m-2,2,6)$ with balance parameter
$$\lambda_6=\frac{(2^m-8)(2^m-16)(2^m-32)}{4!}.$$
\end{thm}
\begin{proof}
Given two distinct elements $u,v \in X$ with $\{u,v\}\not\in W_2,$ denote the number of blocks in $W_6$ that contains both $u$ and $v$ by $\lambda_6(u,v).$  If $z=x+y$ and $S=\{u,v,u+1,v+1\},$ then the blocks sets $\Omega_{z,5},$ $\omega_{\alpha,5},$ and $\tau_{\alpha,5}$ for $\alpha\in S$ can be derived from Definition~\ref{defn_omega_tau}. A block $B\in W_6$ that contains both $x$ and $y$ corresponds to a unique block $\overline{B} \in \Omega_{z,5}$ such that $B\setminus \{u,v\}=\overline{B}\setminus\{z\}.$ For each block $B \in W_6,$ we have
\begin{itemize}
\item[(i)] $|B|=6$ implies $x,y\not\in\overline{B},$
\item[(ii)] ${B \choose 2}\cap W_2=\phi$ implies $u+1,v+1\not\in \overline{B},$ and
\item[(iii)] ${B \choose 3}\cap W_3=\phi$ implies $\not\exists a,b\in \overline{B}\setminus\{z\}$ such that $a+b\in\{u,v,u+1,v+1\}.$
\end{itemize}
Note that the third condition ${B \choose 3}\cap W_3=\phi$ is equivalent to that there do not exist three distinct elements in $B$ with sum in $\{0,1\}.$
Thus, by Lemma~\ref{lem_omega_tau}~(iii) and the cardinalities $|\Omega_{z,5}|,$ $|\omega_{u,5}|,$ and $|\tau_{u,5}|$ given in Lemma~\ref{lem_omega_tau_card}, we have
$$\lambda_6(u,v)=|\Omega_{z,5}|-4|\omega_{u,5}|-2|\tau_{u,5}|
=r_5-4\lambda_5-\frac{1}{2}(2^m-8)(2^m-16).$$
The Venn diagram can be seen in Figure~3. Furthermore, from the values of $\lambda_5$ and $r_5$ respectively given in Theorem~\ref{thm_W5} and Corollary~\ref{cor_W5},
$$\lambda_6(u,v)=\frac{(2^m-8)(2^m-16)(2^m-32)}{4!},$$
which is independent of the choice of $u$ and $v.$ The desired result is obtained.
\end{proof}

\begin{center}
\begin{tikzpicture}
\fill[mygray] (0,0) rectangle (8.5,5);
\fill[white] (1.6,1) circle (0.8)
             (4,1) circle (0.8)
             (6.9,1) circle (0.8)
             (1.6,3.4) circle (0.8)
             (4,3.4) circle (0.8)
             (6.9,3.4) circle (0.8);
\draw (1.6,1) circle (0.8) (1.6,1.8)  node [text=black,above] {$\omega_{u+1,5}$}
      (4,1) circle (0.8) (4,1.8)  node [text=black,above] {$\omega_{v+1,5}$}
      (1.6,3.4) circle (0.8) (1.6,4.2)  node [text=black,above] {$\omega_{u,5}$}
      (4,3.4) circle (0.8) (4,4.2)  node [text=black,above] {$\omega_{v,5}$}
      (6.9,3.4) circle (0.8) (6.9,4.2)  node [text=black,above] {$\tau_{u,5}=\tau_{v+1,5}$}
      (6.9,1) circle (0.8) (6.9,1.8)  node [text=black,above] {$\tau_{v,5}=\tau_{u+1,5}$}
      (0,0) rectangle (8.5,5) node [text=black,above] {$\Omega_{z,5}$};
\end{tikzpicture}\\
\bigskip
{\bf Figure 3:} The Venn diagram for the proof of Theorem~\ref{thm_W6}.
\end{center}

\medskip

Substituting $v=2^m-2,$ $n=2,$ and $k=6$ into~\eqref{eq_r} in Proposition~\ref{prop_r_const} and~\eqref{eq_b} in Proposition~\ref{prop_num_of_block} yields
$$r_6=\lambda_6\cdot\frac{(2^m-2)-2}{6-1}=\lambda_6\cdot\frac{2^m-4}{5}$$
and
$$b_6=r_6\cdot\frac{2^m-2}{6},$$
respectively.
Hence one has Corollary~\ref{cor_W6}.
\begin{cor}     \label{cor_W6}
The repetition number of the triple $(X,W_2,W_6)$ is
$$r_6=\frac{(2^m-4)(2^m-8)(2^m-16)(2^m-32)}{5!}$$
and the number of blocks in $W_6$ is
$$|W_6|=b_6=\frac{(2^m-2)(2^m-4)(2^m-8)(2^m-16)(2^m-32)}{6!}.$$
\qed
\end{cor}

\medskip

The next theorem states a group divisible design with block size $7$. The cardinalities $|\Omega_{z,6}|,$ $|\omega_{\alpha,6}|,$ and $|\tau_{\beta,6}|$ found in Lemma~\ref{lem_omega_tau_card} play an important role in determining the parameter $\lambda_7.$
\begin{thm}     \label{thm_W7}
If $m \geq 7,$ then the triple $(X,W_2,W_7)$ is a $\text{GDD}(2^m-2,2,7)$ with balance parameter
$$\lambda_7=\frac{(2^m-8)(2^m-16)(2^m-32)(2^m-64)}{5!}.$$
\end{thm}
\begin{proof}
Given two distinct elements $u,v \in X$ with $\{u,v\}\not\in W_2,$ let $\lambda_7(u,v)$ be the number of blocks in $W_7$ that contains both $u$ and $v.$ Let $z=u+v$ and $S=\{u,v,u+1,v+1\}.$ A block $B\in W_7$ that contains both $x$ and $y$ corresponds to a unique block $\overline{B} \in \Omega_{z,6}$ such that $B\setminus \{u,v\}=\overline{B}\setminus\{z\},$ where the blocks set $\Omega_{z,6}$ is defined in Definition~\ref{defn_omega_tau} and depicted in Figure~4. For each block $B \in W_6,$ we have
\begin{itemize}
\item[(i)] $|B|=6$ implies $u,v\not\in\overline{B},$
\item[(ii)] ${B \choose 2}\cap W_2=\phi$ implies $u+1,v+1\not\in \overline{B},$
\item[(iii)] ${B \choose 3}\cap W_3=\phi$ implies $\not\exists a,b\in \overline{B}\setminus\{z\}$ such that $a+b\in\{u+1,v+1\},$ and
\item[(iv)] ${B \choose 4}\cap W_4=\phi$ implies $\not\exists a,b\in \overline{B}\setminus\{z\}$ such that $a+b\in\{u,v\}.$
\end{itemize}
Using $\omega_{\alpha,6}$ and $\tau_{\alpha,6}$ for $\alpha\in S=\{u,v,u+1,v+1\}$ in Definition~\ref{defn_omega_tau} and combining all results in Lemma~\ref{lem_omega_tau}~(iv), the cardinalities $|\Omega_{z,6}|,$ $|\omega_{u,6}|$ and $|\tau_{u,6}|$ in Lemma~\ref{lem_omega_tau_card}, the value of $\lambda_6$ in Theorem~\ref{thm_W6}, and the amount of $r_6$ in Corollary~\ref{cor_W6}, the parameter $\lambda_7$ finally becomes
\begin{eqnarray}
\lambda_7(u,v)&=&|\Omega_{z,6}|-4|\omega_{u,6}|-4|\tau_{u,6}|\nonumber\\
&=&r_6-4\lambda_6-\frac{1}{3}(2^m-8)(2^m-16)(2^m-32)\nonumber\\
&=&\frac{(2^m-8)(2^m-16)(2^m-32)(2^m-64)}{5!},\nonumber
\end{eqnarray}
which is independent of the choice of $u$ and $v.$ This completes the proof.
\end{proof}

\begin{center}
\begin{tikzpicture}
\fill[mygray] (0,0) rectangle (11,5);
\fill[white] (1.6,1) circle (0.8)
             (4,1) circle (0.8)
             (7,1) circle (0.8)
             (1.6,3.4) circle (0.8)
             (4,3.4) circle (0.8)
             (7,3.4) circle (0.8)
             (9.4,1) circle (0.8)
             (9.4,3.4) circle (0.8);
\draw (1.6,1) circle (0.8) (1.6,1.8)  node [text=black,above] {$\omega_{u+1,6}$}
      (4,1) circle (0.8) (4,1.8)  node [text=black,above] {$\omega_{v+1,6}$}
      (1.6,3.4) circle (0.8) (1.6,4.2)  node [text=black,above] {$\omega_{u,6}$}
      (4,3.4) circle (0.8) (4,4.2)  node [text=black,above] {$\omega_{v,6}$}
      (7,3.4) circle (0.8) (7,4.2)  node [text=black,above] {$\tau_{u,6}$}
      (7,1) circle (0.8) (7,1.8)  node [text=black,above] {$\tau_{u+1,6}$}
      (9.4,3.4) circle (0.8) (9.4,4.2)  node [text=black,above] {$\tau_{v,6}$}
      (9.4,1) circle (0.8) (9.4,1.8)  node [text=black,above] {$\tau_{v+1,6}$}
      (0,0) rectangle (11,5) node [text=black,above] {$\Omega_{z,6}$};
\end{tikzpicture}\\
\bigskip
{\bf Figure 4:} The Venn diagram for the proof of Theorem~\ref{thm_W7}.
\end{center}

\medskip

Corollary~\ref{cor_W7} is easily carried out according to the formulas $r=\lambda(v-n)/(k-1)$ in Proposition~\ref{prop_r_const} and $b=vr/k$ in Proposition~\ref{prop_num_of_block}.
\begin{cor}     \label{cor_W7}
The repetition number of the triple $(X,W_2,W_7)$ is
$$r_7=\frac{(2^m-4)(2^m-8)(2^m-16)(2^m-32)(2^m-64)}{6!}$$
and the number of blocks in $W_7$ is
$$|W_7|=b_7=\frac{(2^m-2)(2^m-4)(2^m-8)(2^m-16)(2^m-32)(2^m-64)}{7!}.$$
\qed
\end{cor}


\section{Concluding remark}\label{sec_con}

This paper has demonstrated that the triple $(X,W_{2},W_{k})$ is a $\text{GDD}(2^{m}-2,2,k)$ for $k=3,4,5,6,7$ and $m \geq k.$ The balance parameter $\lambda_k,$ repetition number $r_k$ and number of blocks $b_k$ of each GDD are shown in Table 1.
\renewcommand{\arraystretch}{2.5}
\begin{center}
{\bf Table 1:} The balance parameter $\lambda_k,$ repetition number $r_k$ and number of blocks $b_k$ of triple $(X,W_{2},W_{k})$ for $k=3,4,5,6,$ and $7,$ respectively.

\bigskip
\begin{tabular}{|c|c|c|c|}
\hline
      & $\lambda_{k}$ & $r_{k}$ &  $b_{k}$ \\  \hline
$k=3$ & $1$     & $\frac{2^{m}-4}{2!}$  & $\frac{(2^{m}-2)(2^{m}-4)}{3!}$ \\ \hline
$k=4$ & $\frac{2^{m}-8}{2!}$ & $\frac{(2^{m}-4)(2^{m}-8)}{3!}$ & $\frac{(2^{m}-2)(2^{m}-4)(2^{m}-8)}{4!}$ \\ \hline
$k=5$ & $\frac{(2^{m}-8)(2^{m}-16)}{3!}$ & $\frac{(2^{m}-4)(2^{m}-8)(2^{m}-16)}{4!}$ & $\frac{(2^{m}-2)(2^{m}-4)(2^{m}-8)(2^{m}-16)}{5!}$ \\ \hline
$k=6$ & $\frac{\prod\limits_{i=3}^{5} (2^m-2^i)}{4!}$ & $\frac{\prod\limits_{i=2}^{5} (2^m-2^i)}{5!}$ & $\frac{\prod\limits_{i=1}^{5} (2^m-2^i)}{6!}$ \\ \hline
$k=7$ & $\frac{\prod\limits_{i=3}^{6} (2^m-2^i)}{5!}$ & $\frac{\prod\limits_{i=2}^{6} (2^m-2^i)}{6!}$ & $\frac{\prod\limits_{i=1}^{6} (2^m-2^i)}{7!}$
\\ \hline
\end{tabular}\\
\end{center}

\medskip

By observing the above table, we give a conjecture that the triple $(X,W_{2},W_{k})$ is a GDD for all $m \geq k \geq 3$ including the exact values for the parameters $\lambda_{k},$ $r_{k},$ and $b_{k}.$
\begin{conj} \label{conj_con_rem}
For $m \geq k \geq 3,$ let
$$W_k=\left\{\begin{array}{c|c}
 \{x_1,x_2,\ldots,x_k\}\subset X  & \sum\limits_{i=1}^{k}x_i=1, \text{ and }\sum\limits_{i\in I}x_i\neq 1\text{ for each } \\
   & \text{ nonempty proper subset } I\subset\{1,2,\ldots,k\}
\end{array}\right\}.$$
Then, the triple $(X,W_{2},W_{k})$ is a $\text{GDD}(2^{m}-2,2,k)$ with
\begin{eqnarray}
\text{balance parameter}~~~~~~\lambda_{k}&=&\frac{\prod\limits_{i=3}^{k-1} (2^m-2^i)}{(k-2)!},
\nonumber \\
\text{repetition number}~~~~~~r_{k}&=&\frac{\prod\limits_{i=2}^{k-1} (2^m-2^i)}{(k-1)!},~~~\text{and}
\nonumber \\
\text{number of blocks}~~~~~~b_{k}&=&\frac{\prod\limits_{i=1}^{k-1} (2^m-2^i)}{k!}.
\nonumber
\end{eqnarray}
\end{conj}

\medskip

The cases $k\leq 7$ of Conjecture~\ref{conj_con_rem} has been proved in this paper by using the including-excluding principle. Due to the complication for larger $k,$ the key behind the proof might contain other counting methods.



\section*{Acknowledgments}
This research is supported by the Ministry of Science and Technology of Taiwan R.O.C. under the project MOST 103-2632-M-214-001-MY3-2 including its subproject MOST 104-2811-M-214-001.

\bigskip

\noindent Chia-an Liu \hfil\break
Department of Financial and Computational Mathematics \hfil\break
I-Shou University \hfil\break
No.1 Sec.1 Xuecheng Rd. Dashu Dist. \hfil\break
Kaohsiung, Taiwan 84001 R.O.C. \hfil\break
Email: {\tt liuchiaan8@gmail.com} \hfil\break
Ext: +886-7-6577711-5612 \hfil\break

\end{document}